\documentclass[12pt,reqno]{amsart}

\usepackage{multirow}
\usepackage{hhline}
\usepackage{float}

\usepackage{mathtools}
\usepackage{times}
\usepackage[T1]{fontenc}
\usepackage{mathrsfs}
\usepackage{latexsym}
\usepackage[dvips]{graphics}
\usepackage[titletoc, title]{appendix}
\usepackage{amsmath,amsfonts,amsthm,amssymb,amscd}
\usepackage[dvipsnames]{xcolor}
\usepackage{hyperref}
\usepackage{amsmath}
\usepackage[utf8]{inputenc}

\usepackage{color}
\usepackage{breakurl}

\usepackage{comment}
\newcommand{\bburl}[1]{\textcolor{blue}{\url{#1}}}
\newcommand{\seqnum}[1]{\href{https://oeis.org/#1}{\rm \underline{#1}}}

\usepackage{caption}

\newtheorem{thm}{Theorem}[section]

\newtheorem{lem}[thm]{Lemma}
\newtheorem{prop}[thm]{Proposition}

\newtheorem{defi}[thm]{Definition}
\newtheorem{rek}[thm]{Remark}
\newtheorem{prob}[thm]{Problem}

\usepackage[utf8]{inputenc}

\DeclareFixedFont{\ttb}{T1}{txtt}{bx}{n}{12} % for bold
\DeclareFixedFont{\ttm}{T1}{txtt}{m}{n}{12}  % for normal

\usepackage{color}
\definecolor{deepblue}{rgb}{0,0,0.5}
\definecolor{deepred}{rgb}{0.6,0,0}
\definecolor{deepgreen}{rgb}{0,0.5,0}

\usepackage{listings}

\newcommand\pythonstyle{\lstset{
language=Python,
basicstyle=\ttm,
morekeywords={self},              % Add keywords here
keywordstyle=\ttb\color{deepblue},
emph={MyClass,__init__},          % Custom highlighting
emphstyle=\ttb\color{deepred},    % Custom highlighting style
stringstyle=\color{deepgreen},
frame=tb,                         % Any extra options here
showstringspaces=false
}}

\lstnewenvironment{python}[1][]
{
\pythonstyle
\lstset{#1}
}
{}

\newcommand\pythoninline[1]{{\pythonstyle\lstinline!#1!}}

\definecolor{ao}{rgb}{0.0, 0.5, 0.0}

\numberwithin{equation}{section}

\DeclareFontFamily{U}{mathx}{}
\DeclareFontShape{U}{mathx}{m}{n}{<-> mathx10}{}
\DeclareSymbolFont{mathx}{U}{mathx}{m}{n}
\DeclareMathAccent{\widehat}{0}{mathx}{"70}
\DeclareMathAccent{\widecheck}{0}{mathx}{"71}

\begin{document}

\title[Linear Recurrences from Counting Schreier-Type Multisets]{Linear Recurrences from Counting Schreier-Type Multisets}

\author[H. V. Chu]{H\`ung Vi\d{\^e}t Chu}
\email{\textcolor{blue}{\href{mailto:hchu@wlu.edu}{hchu@wlu.edu}}}
\address{Department of Mathematics, Washington and Lee University, Lexington, VA 24450, USA}   

\author[Y. Geng]{Yubo Geng}
\email{\textcolor{blue}{\href{mailto:geng0114@umn.edu}{geng0114@umn.edu}}}
\address{College of Liberal Arts, University of Minnesota, Minneapolis, MN 55455, USA}  

\author[J. King]{Julian King}
\email{\textcolor{blue}{\href{mailto:jhk12@geneseo.edu}{jhk12@geneseo.edu}}}
\address{Department of Mathematics, SUNY Geneseo, Geneseo, NY 14454, USA} 

\author[S. J. Miller]{Steven J. Miller}
\email{\textcolor{blue}{\href{mailto:sjm1@williams.edu}{sjm1@williams.edu},
\href{mailto:Steven.Miller.MC.96@aya.yale.edu}{Steven.Miller.MC.96@aya.yale.edu}}}
\address{Department of Mathematics and Statistics, Williams College, Williamstown, MA 01267, USA}

\author[G. Tresch]{Garrett Tresch}
\email{\textcolor{blue}{\href{mailto:treschgd@tamu.edu}{treschgd@tamu.edu}}}
\address{Department of Mathematics, Texas A\&M University, College Station, TX 77840, USA}

\author[Z. L. Vasseur]{Zachary Louis Vasseur}
\email{\textcolor{blue}{\href{mailto:zachary.l.v@tamu.edu}{zachary.l.v@tamu.edu}}}
\address{Department of Mathematics, Texas A\&M University, College Station, TX 77840, USA}

\thanks{This work was partially supported by the National Science Foundation DMS2341670. We thank the
participants at Polymath Jr. 2025 for helpful discussions.}



\maketitle
 
\begin{abstract}
A nonempty set $F$ is Schreier if $\min F\ge |F|$. Bird observed that counting Schreier sets in a certain way produces the Fibonacci sequence. Since then, various connections between variants of Schreier sets and well-known sequences have been discovered. Building on these works, we prove a linear recurrence for the sequence that counts multisets $F$ with $\min F\ge p|F|$. In particular, if we let 
$$\mathcal{A}^{(s)}_{p, n}\ :=\ \{F\subset \{\underbrace{1, \ldots, 1}_{s}, \ldots, \underbrace{n-1, \ldots, n-1}_{s}, n\}\,:\,n\in F\mbox{ and }\min F\ge p|F|\},$$
then $$|\mathcal{A}^{(s)}_{p, n}| \ =\ \sum_{i=0}^s|\mathcal{A}^{(s)}_{p, n-1-ip}|.$$
If we color $s$ copies of the same integer by different colors from $1$ to $s$, i.e., $\mathcal{B}^{(s)}_{p, n}:= $
$$\{F\subset \{1_{1}, \ldots, 1_{s}, \ldots, (n-1)_1, \ldots, (n-1)_{s}, n\}\,:\,n\in F\mbox{ and }\min F\ge p|F|\},$$
then 
$$|\mathcal{B}^{(s)}_{p, n}| \ =\ \sum_{i=0}^s\binom{s}{i}|\mathcal{B}^{(s)}_{p, n-1-ip}|.$$
Lastly, we count Schreier sets that do not admit multiples of a given integer $u\ge 2$ and witness linear recurrences whose coefficients are drawn from the $u$\textsuperscript{th} row of the Pascal triangle and have alternating signs, except possibly the last one. 
\end{abstract}

\tableofcontents

\section{Introduction}

A finite, nonempty set $F\subset \mathbb{N}$ is \textit{Schreier} if $\min F\ge |F|$. Bird \cite{Bi} discovered a fascinating connection between Schreier sets and Fibonacci numbers: for $n\in \mathbb{N}$,
$$|\{F\subset \{1, 2, \ldots, n\}: F\mbox{ is Schreier and }n\in F\}|\ =\ F_n,$$
where $(F_n)_{n=1}^\infty$ is the Fibonacci sequence with $F_1 = F_2 = 1$ and $F_{n} = F_{n-1} + F_{n-2}$ for $n\ge 3$. 
If instead, we consider Schreier multisets, \cite[Theorem 1]{CIMSZ} states that
$$|\{F\subset \{\underbrace{1,\ldots, 1}_{s}, \ldots, \underbrace{n-1, \ldots, n-1}_s, n\}\,:\, F\mbox{ is Schreier and }n\in F\}|\ =\ F^{(s+1)}_n,$$
where for each $m\ge 2$, $(F^{(m)}_n)_{n=1}^\infty$ is the $m$-step Fibonacci sequence with 
\begin{align*}
    F^{(m)}_{2-m} &\ =\ \cdots \ =\ F_0^{(m)} \ =\ 0, \quad F^{(m)}_1 \ =\ 1, \quad \mbox{ and }\\
    F^{(m)}_{n} &\ =\ F^{(m)}_{n-1} + F^{(m)}_{n-2} + \cdots + F^{(m)}_{n-m}, \mbox{ for }n\ge 2.
\end{align*}
These $m$-step Fibonacci sequences are listed in The On-Line Encyclopedia of Integer Sequences \cite{Sl} as \seqnum{A000045}, \seqnum{A000073}, \seqnum{A000078}, and so on. 
For more connections between Schreier-type sets and various sequences, the readers may refer to  
\cite{BC, BCF, BGHH, CMX, CV}.

Inspired by \cite{BCF, CIMSZ}, we investigate the following general problem of counting Schreier-type multisets. 

\begin{prob}\normalfont\label{mainprob}
  Given $(p,q)\in \mathbb{N}^2$ and for each $n\in \mathbb{N}$, a sequence of nonnegative integers $(s_{n,i})_{i=1}^n$, define
    $$H^{(s_{n,i})}_{p, q, n}\ :=\ \{F\subset \{\underbrace{1, \ldots, 1}_{s_{n,1}}, \ldots, \underbrace{n, \ldots, n}_{s_{n,n}}\}\,:\,n\in F\mbox{ and }q\min F\ge p|F|\}.$$
  Find the recurrence (if any) of $(|H^{(s_{n,i})}_{p, q, n}|)_{n=1}^\infty$.
\end{prob}

We solve Problem \ref{mainprob} and its colored version when $q = 1$, $s_{n,1} = \cdots = s_{n, n-1} = s$ for some fixed positive integer $s$, and $s_{n,n} = 1$. Our main results show that the sequences from counting these multisets both satisfy linear recurrences whose indices are in arithmetic progression as well as linear recurrences with coefficients taken from the Pascal's triangle, respectively. For $(s, p, n)\in \mathbb{N}^3$, define
$$\mathcal{A}^{(s)}_{p, n}\ :=\ \{F\subset \{\underbrace{1, \ldots, 1}_{s}, \ldots, \underbrace{n-1, \ldots, n-1}_{s}, n\}\,:\,n\in F\mbox{ and }\min F\ge p|F|\}.$$
Table \ref{skn} gives sample data for $(|\mathcal{A}^{(s)}_{p, n}|)_{n=1}^\infty$.

\begin{table}[H]
\centering
\begin{tabular}{ |c| c| c| c| c| c| c| c| c| c| c| c| c| c| c| c| c| c|}
\hline
$n$ &$1$& $2$ & $3$ & $4$ & $5$ & $6$ & $7$ & $8$ & $9$ & $10$ & $11$ & $12$  \\
\hline
$|\mathcal{A}^{(1)}_{1, n}| (\seqnum{A000045})$ & $1$ & $1$ & $2$ & $3$ & $5$ & $8$ & $13$ & $21$ & $34$ & $55$ & $89$ & $144$  \\
\hline
$|\mathcal{A}^{(1)}_{2, n}| (\seqnum{A078012})$ & $0$ & $1$ & $1$ & $1$ & $2$ & $3$ & $4$ & $6$ & $9$ & $13$ & $19$ & $28$  \\
\hline
$|\mathcal{A}^{(2)}_{1, n}| (\seqnum{A000073})$ & $1$ & $1$ & $2$ & $4$ & $7$ & $13$ & $24$ & $44$ & $81$ & $149$ & $274$ & $504$ \\
\hline
$|\mathcal{A}^{(2)}_{2, n}| (\seqnum{A060961})$ & $0$ & $1$ & $1$ & $1$ & $2$ & $3$ & $5$ & $8$ & $12$ & $19$ & $30$ & $47$  \\
\hline
$|\mathcal{A}^{(3)}_{1, n}| (\seqnum{A000078})$ & $1$ & $1$ & $2$ & $4$ & $8$ & $15$ & $29$ & $56$ & $108$ &  $208$ &  $401$ &  $773$ \\
\hline
$|\mathcal{A}^{(3)}_{2, n}| (\seqnum{A117760})$ & $0$ & $1$ & $1$ & $1$ & $2$ & $3$ & $5$ & $8$ & $13$ & $21$ & $33$ & $53$\\
\hline
\end{tabular}
\caption{The first $12$ values of $(|\mathcal{A}^{(s)}_{p, n}|)_{n=1}^\infty$ with $(s,p) \in \{(1,1), (1, 2), (2,1), (2,2), (3,1), (3,2)\}$.}
\label{skn}
\end{table}

\begin{thm}\label{m1}
For $1\le n\le sp + 1$, 
\begin{equation}\label{e3}|\mathcal{A}^{(s)}_{p, n}| \ =\ \sum_{k=1}^{\left\lfloor \frac{sn+1}{sp+1}\right\rfloor}\binom{n-pk+k-2}{k-1}.\end{equation}
For $n\ge sp + 2$, we have
\begin{equation}\label{e4}|\mathcal{A}^{(s)}_{p, n}|\ =\ \sum_{i=0}^s|\mathcal{A}^{(s)}_{p, n-1-ip}|.\end{equation}
\end{thm}

\begin{rek}\normalfont
   It is interesting to see that $p$ and $s$ have different effects on Recurrence \eqref{e4}. While $(s+1)$ gives the number of terms in the recurrence, $p$ equals the gap between consecutive indices. 
\end{rek}

\begin{rek}\normalfont
    The problem of extending Theorem \ref{m1} to all $q\in\mathbb{N}$ (instead of $q = 1$) remains open and is discussed in the last section.
\end{rek}

Next, we give a colored version of Theorem \ref{m1}. Given an integer $k$, we color the $s$ copies of $k$ using  $s$ different colors and denote them by $k_1, k_2, \ldots, k_s$. For example, while $k_1$ and $k_2$ have the same numerical value, the two sets $\{k_1\}$ and $\{k_2\}$ are distinguisable. Let 
$$\mathcal{B}^{(s)}_{p, n}\ :=\ \{F\subset \{1_{1}, \ldots, 1_{s}, \ldots, (n-1)_1, \ldots, (n-1)_{s}, n\}\,:\,n\in F\mbox{ and }\min F\ge p|F|\}.$$
We collect sample data for $(|\mathcal{B}^{(s)}_{p, n}|)_{n=1}^\infty$ in Table \ref{Bn}.

\begin{table}[H]
\centering
\begin{tabular}{ |c| c| c| c| c| c| c| c| c| c| c| c| c| c| c| c| c| c|}
\hline
$n$ &$1$& $2$ & $3$ & $4$ & $5$ & $6$ & $7$ & $8$ & $9$ & $10$ & $11$ & $12$  \\
\hline
$|\mathcal{B}^{(1)}_{1, n}| (\seqnum{A000045})$ & $1$ & $1$ & $2$ & $3$ & $5$ & $8$ & $13$ & $21$ & $34$ & $55$ & $89$ & $144$  \\
\hline
$|\mathcal{B}^{(1)}_{2, n}| (\seqnum{A078012})$ & $0$ & $1$ & $1$ & $1$ & $2$ & $3$ & $4$ & $6$ & $9$ & $13$ & $19$ & $28$  \\
\hline
$|\mathcal{B}^{(2)}_{1, n}| (\seqnum{A002478})$ & $1$ & $1$ & $3$ & $6$ & $13$ & $28$ & $60$ & $129$ & $277$ & $595$ & $1278$ & $2745$ \\
\hline
$|\mathcal{B}^{(2)}_{2, n}| (\seqnum{A193147})$ & $0$ & $1$ & $1$ & $1$ & $3$ & $5$ & $8$ & $15$ & $26$ & $45$ & $80$ & $140$  \\
\hline
$|\mathcal{B}^{(3)}_{1, n}| (\seqnum{A099234})$ & $1$ & $1$ & $4$ & $10$ & $26$ & $69$ & $181$ & $476$ & $1252$ &  $3292$ & $8657$  &  $22765$ \\
\hline
$|\mathcal{B}^{(3)}_{2, n}| (\seqnum{A373718})$ & $0$ & $1$ & $1$ & $1$ & $4$ & $7$ & $13$ & $28$ & $53$ & $105$ & $211$ & $413$\\
\hline
\end{tabular}
\caption{The first $12$ values of $(|\mathcal{B}^{(s)}_{p, n}|)_{n=1}^\infty$ with $(s,p) \in \{(1,1), (1, 2), (2,1), (2,2), (3,1), (3,2)\}$.}
\label{Bn}
\end{table}

\begin{thm}\label{m2} Let $s$ and $p$ be positive integers. For $1\le n\le sp+1$,
$$|\mathcal{B}^{(s)}_{p,n}|\ =\ \sum_{k=1}^{\left\lfloor \frac{ns+1}{ps+1}\right\rfloor}\binom{(n-kp)s}{k-1}.$$
For $n\ge sp+2$, we have
    \begin{align}\label{kr}
        |\mathcal{B}^{(s)}_{p, n}|&\ =\ \binom{s}{0}|\mathcal{B}^{(s)}_{p, n-1}| + \binom{s}{1}|\mathcal{B}^{(s)}_{p, n-1-p}| + \cdots + \binom{s}{s}|\mathcal{B}^{(s)}_{p, n-1-sp}|\nonumber\\
        &\  =\ \sum_{i=0}^s\binom{s}{i}|\mathcal{B}^{(s)}_{p, n-1-ip}|. 
    \end{align}     
\end{thm}

Motivated by a recent result in \cite{CV} that counts Schreier sets of multiples of a fixed $u\ge 2$, our last result counts Schreier sets that do not admit any multiple of $u$. If $g_u(n)$ denotes the $n$\textsuperscript{th} positive integer not divisible by $u$, then as we show later, $g_u(n) = \left\lfloor(un-1)/(u-1)\right\rfloor$.
Let $G_{u,n}$ be the set of the first $n$ positive integers that are not divisible by $u$ and define
$$\mathcal{D}_{u, n}\ :=\ \left\{F\subset G_{u,n}\,:\, \left\lfloor\frac{un-1}{u-1}\right\rfloor\in F\mbox{ and }F\mbox{ is Schreier}\right\}.$$

\begin{table}[H]
\centering
\begin{tabular}{ |c| c| c| c| c| c| c| c| c| c| c| c| c| c| c| c| c| c|}
\hline
$n$ &$1$& $2$ & $3$ & $4$ & $5$ & $6$ & $7$ & $8$ & $9$ & $10$ & $11$ & $12$  \\
\hline
$|\mathcal{D}_{2, n}| (\seqnum{A005251})$ & $1$ & $1$ & $2$ & $4$ & $7$ & $12$ & $21$ & $37$ & $65$ & $114$ & $200$ & $351$  \\
\hline
$|\mathcal{D}_{3, n}| (\seqnum{A137402})$ & $1$ & $1$ & $2$ & $3$ & $5$ & $9$ & $16$ & $28$ & $48$ & $81$ & $136$ & $229$  \\
\hline
$|\mathcal{D}_{4, n}|$ & $1$ & $1$ & $2$ & $3$ & $5$ & $8$ & $13$ & $22$ & $38$ & $66$ & $114$ & $195$ \\
\hline
$|\mathcal{D}_{5, n}|$ & $1$ & $1$ & $2$ & $3$ & $5$ & $8$ & $13$ & $21$ & $34$ & $56$ & $94$ & $160$  \\
\hline
$|\mathcal{D}_{6, n}|$ & $1$ & $1$ & $2$ & $3$ & $5$ & $8$ & $13$ & $21$ & $34$ &  $55$ & $89$  &  $145$ \\
\hline
\end{tabular}
\caption{The first $12$ values of $(|\mathcal{D}_{u, n}|)_{n=1}^\infty$ with $u \in \{2, \ldots, 6\}$.}
\label{Dn}
\end{table}

Table \ref{Dn} suggests the following recurrence. 

\begin{thm}\label{m3}
For $1\le n\le 2u-1$, we have $|\mathcal{D}_{u, n}| = F_n$. 
    For $n\ge 2u$, 
    $$|\mathcal{D}_{u, n}| \ =\ \sum_{i=1}^{u}(-1)^{i-1}\binom{u}{i}|\mathcal{D}_{u, n-i}| + |\mathcal{D}_{u, n - 2u+1}|.$$
\end{thm}

We devote Sections \ref{uncolored}, \ref{colored}, and \ref{notmul} to the proof of Theorems \ref{m1}, \ref{m2}, and \ref{m3}, respectively. Section \ref{future} discusses directions for future investigations. To keep our proofs concise and emphasize the main ideas, we relegate certain technical proofs to Section \ref{apen}.

\section{Uncolored generalized Schreier multisets}\label{uncolored}

We give a formula to compute $|\mathcal{A}^{(s)}_{p, n}|$, which is used to prove the first statement of Theorem \ref{m1}, while we prove the second statement using bijective maps. For $k\ge 1$, each $k$-element set in $\mathcal{A}^{(s)}_{p, n}$ is uniquely determined by $(k-1)$ numbers from the multiset 
$$\{\underbrace{pk, \ldots, pk}_s, \ldots, \underbrace{n-1, \ldots, n-1}_s\}.$$
Let $\binom{n}{m}_s$ be the number of ways to distribute $m$ identical objects among $n$ labeled boxes, each of which contains at most $s-1$ objects. According to \cite[(1.16)]{Bon},
$$\binom{n}{m}_s\ =\ \sum_{k=0}^{\lfloor m/s\rfloor} (-1)^k\binom{n}{k}\binom{n+m-sk-1}{n-1}.$$
It follows that the number of $k$-element sets in $\mathcal{A}^{(s)}_{p, n}$ is $\binom{n-pk}{k-1}_{s+1}$ under the condition
$$s(n-pk)\ \ge\ k-1\ \Longrightarrow\ k\ \le\ \frac{sn+1}{sp+1}.$$
Therefore,
\begin{align}\label{e2}|\mathcal{A}^{(s)}_{p,n}|
&\ =\  \sum_{k=1}^{\left\lfloor \frac{sn+1}{sp+1}\right\rfloor}\binom{n-pk}{k-1}_{s+1}\nonumber\\
&\ =\ \sum_{k=1}^{\left\lfloor \frac{sn+1}{sp+1}\right\rfloor}\sum_{j=0}^{\left\lfloor\frac{k-1}{s+1}\right\rfloor}(-1)^j\binom{n-pk}{j}\binom{n-pk+k-(s+1)j-2}{n-pk-1}.\end{align}

\begin{proof}[Proof of Theorem \ref{m1}]
We note
$$\frac{k-1}{s+1}\ \le\ \frac{\frac{sn+1}{sp+1}-1}{s+1}\ \le\ \frac{\frac{s(sp+1)+1}{sp+1}-1}{s+1}\ =\ \frac{s+\frac{1}{sp+1}-1}{s+1}\ <\ 1.$$
Hence, it follows from \eqref{e2} that
\begin{align*}
|\mathcal{A}^{(s)}_{p,n}|&\ =\ \sum_{k=1}^{\left\lfloor \frac{sn+1}{sp+1}\right\rfloor}\sum_{j=0}^{\left\lfloor\frac{k-1}{s+1}\right\rfloor}(-1)^j\binom{n-pk}{j}\binom{n-pk+k-(s+1)j-2}{n-pk-1}\\
&\ =\ \sum_{k=1}^{\left\lfloor \frac{sn+1}{sp+1}\right\rfloor}\binom{n-pk+k-2}{n-pk-1}\\
&\ =\ \sum_{k=1}^{\left\lfloor \frac{sn+1}{sp+1}\right\rfloor}\binom{n-pk+k-2}{k-1}.
\end{align*}

To prove \eqref{e4}, we partition $\mathcal{A}^{(s)}_{p,n}$ into $(s+1)$ sets $(\mathcal{A}^{(s)}_{p,n, i})_{i=0}^s$, where 
$$\mathcal{A}^{(s)}_{p,n, i}\ :=\ \{F\in \mathcal{A}^{(s)}_{p,n}\,:\, F\mbox{ contains exactly } i \mbox{ copies of } (n-1)\}.$$
For each integer $i\in [0, s]$, define a map $\psi_i: \mathcal{A}^{(s)}_{p, n-1-ip}\rightarrow \mathcal{A}^{(s)}_{p, n, i}$ as
$$F\ \longrightarrow\ \begin{cases} (F\backslash \{n-1\})\cup \{n\}, &\mbox{ if } i = 0;\\
\left((F+ip)\cup\{\underbrace{n-1,\ldots, n-1}_{i-1}\}\right)\cup\{n\}, &\mbox{ if } 1\le i\le s.
\end{cases}$$

First, we show that each $\psi_i$ is well-defined.
\begin{enumerate}
    \item[a)] When $i = 0$, $\psi_0$ increases the maximum element of the input set by $1$, so if $F\in \mathcal{A}^{(s)}_{ p, n-1}$, then $\psi_0(F)$ does not contain $(n-1)$, contains exactly one copy of $n$ as the maximum, and $$\min\psi_0(F) \ \ge\ \min F\ \ge\ p|F| \ =\ p|\psi_0(F)|.$$
Hence, $\psi_0(F)\in \mathcal{A}^{(s)}_{p,n,0}$.
    \item[b)] When $1\le i\le s$, $\psi_i$ increases the size of the input set by $i$ and increases the minimum by $ip$, so
    $$\min\psi_i(F) \ =\ \min F + ip\ \ge\ p|F|+ip \ =\ p|\psi_i(F)|.$$
    Furthermore, $\psi_i(F)$ contains exactly $i$ copies of $n-1$ and contains exactly one copy of $n$ as the maximum. Hence,
    $\psi_i(F)\in\mathcal{A}^{(s)}_{p,n, i}$.
\end{enumerate}

Next, we prove that each $\psi_i$ is a bijection. Injectivity follows immediately from the definition of $\psi_i$. We show that each $\psi_i$ is onto. 
\begin{enumerate}
    \item[a)] Let $G\in \mathcal{A}^{(s)}_{p, n, 0}$. Then $n\in G$ but $(n-1)\notin G$. Let $F := (G\cup \{n-1\})\backslash \{n\}$. We have $F$ contain exactly one copy of $(n-1)$ as the maximum, and 
    $$\min F\ =\ \begin{cases}\min G\ \ge\ p|G| \ =\ p|F|,&\mbox{ if }|G| > 1;\\
     n - 1\ \ge\ p\ =\ p|F|, &\mbox{ if } G = \{n\} \mbox{ (because $n\ge sp+2$)}.
    \end{cases}$$
    Thus, $F\in \mathcal{A}^{(s)}_{p, n-1}$ and $\psi_0(F) = G$.

    \item[b)] Let $G\in \mathcal{A}^{(s)}_{p,n,i}$ with $1\le i\le s$. Let 
    $$F \ :=\ (G\backslash \{n, \underbrace{n-1, \ldots, n-1}_{i-1}\}) - ip.$$
    Since $G$ has exactly $i$ copies of $(n-1)$, the set $F$ contains exactly one copy of $(n-1-ip)$ as the maximum. Furthermore,
    $$\min F\ =\ \min G - ip\ \ge\ p|G| - ip\ =\ p(|G| - i)\ =\ p|F|.$$
    Thus, $F\in \mathcal{A}^{(s)}_{p, n-1-ip}$ and $\psi_i(F) = G$.
\end{enumerate}
We have shown that $|\mathcal{A}^{(s)}_{p, n-1-ip}| = |\mathcal{A}^{(s)}_{p, n,i}|$. Therefore,
$$|\mathcal{A}^{(s)}_{p , n}|\ =\ \sum_{i=0}^s |\mathcal{A}^{(s)}_{p, n,i}|\ =\ \sum_{i=0}^s |\mathcal{A}^{(s)}_{p, n-1-ip}|.$$
\end{proof}

\section{Colored generalized Schreier multisets}\label{colored}
We find a formula for $|\mathcal{B}^{(s)}_{p, n}|$. To form a $k$-element set, we choose $(k-1)$ elements in  
$$\{(kp)_{1}, \ldots, (kp)_{s}, \ldots, (n-1)_1, \ldots, (n-1)_s\}.$$
Hence, the number of $k$-element sets in $\mathcal{B}^{(s)}_{p,n}$ is $\binom{(n-kp)s}{k-1}$
with $k-1 \le (n-kp)s$, i.e., $k\le (ns+1)/(ps+1)$.
Therefore, 
\begin{equation}\label{e20}|\mathcal{B}^{(s)}_{p, n}|\ =\ \sum_{k=1}^{\left\lfloor \frac{ns+1}{ps+1}\right\rfloor}\binom{(n-kp)s}{k-1}.\end{equation}
This proves the first statement of Theorem \ref{m2}. We use characteristic polynomials to prove the second statement. The first step is to find a parent sequence, denoted by $(b^{(s)}_{p, n})_{n=1}^\infty$, that has $(|\mathcal{B}^{(s)}_{p,n}|)_{n=1}^\infty$ as a periodic subsequence. The parent sequence $(b^{(s)}_{p, n})_{n=1}^\infty$ should satisfy a relatively simple recurrence that we take advantage of later. 

Given positive integers $s$ and $p$, we define the sequence $(b^{(s)}_{p, n})_{n=1}^\infty$ recursively as follows:
$$b^{(s)}_{p, 1} \ =\ \cdots \ =\ b^{(s)}_{p, (p-1)s} \ =\ 0,\quad b^{(s)}_{p, (p-1)s+1} \ =\ \cdots\ =\ b^{(s)}_{p, ps + 1}\ =\ 1, \quad \mbox{ and }$$
$$b^{(s)}_{p, n}\ =\ b^{(s)}_{p, n-1} + b^{(s)}_{p,n-1-ps}, \mbox{ for }n\ge ps+2.$$
For example, 
\begin{align*}
&(b^{(2)}_{1, n})_{n=1}^\infty (\seqnum{A000930}): 1, 1, 1, 2, 3, 4, 6, 9, 13, 19, 28, 41, 60, 88, 129, 189, 277, 406,\ldots;\\
&(b^{(2)}_{2, n})_{n=1}^\infty (\seqnum{A003520}): 0, 0, 1, 1, 1, 1, 1, 2, 3, 4, 5, 6, 8, 11, 15, 20, 26, 34, 45, 60, \ldots;\\
&(b^{(3)}_{2, n})_{n=1}^\infty (\seqnum{A005709}): 0, 0, 0, 1, 1, 1, 1, 1, 1, 1, 2, 3, 4, 5, 6, 7, 8, 10, 13, 17, 22, 28, 35, \ldots.
\end{align*}

\begin{prop}\label{formulaparentseq} For $(s,p, n)\in \mathbb{N}^3$, it holds that
$$b^{(s)}_{p, n} \ =\ \sum_{i=0}^{\left\lfloor \frac{n-s(p-1)-1}{sp+1}\right\rfloor}\binom{n-s(p-1)-1-spi}{i}.$$
\end{prop}

\begin{proof}
For $n\le (p-1)s$, we have
$$\frac{n-s(p-1)-1}{sp+1}\ \le\ \frac{-1}{sp+1},\mbox{ so }\left\lfloor \frac{n-s(p-1)-1}{sp+1}\right\rfloor\ \le\ -1.$$
Hence
$$\sum_{i=0}^{\left\lfloor \frac{n-s(p-1)-1}{sp+1}\right\rfloor}\binom{n-s(p-1)-1-spi}{i} \ =\ 0,$$
as desired.

For $(p-1)s+1\le n\le ps+1$, we have
$$\frac{n-s(p-1)-1}{sp+1}\ \ge\ \frac{(p-1)s+1-s(p-1)-1}{sp+1}\ =\ 0$$
and 
$$\frac{n-s(p-1)-1}{sp+1}\ \le\ \frac{ps+1-s(p-1)-1}{sp+1}\ =\ \frac{s}{sp+1} \ <\ 1.$$
Thus
$$\left\lfloor \frac{n-s(p-1)-1}{sp+1}\right\rfloor\ =\ 0,$$
which gives
$$\sum_{i=0}^{\left\lfloor \frac{n-s(p-1)-1}{sp+1}\right\rfloor}\binom{n-s(p-1)-1-spi}{i} \ =\ \binom{n-s(p-1)-1}{0}\ =\ 1.$$

For $n\ge ps+2$, we show that
\begin{align*}
\sum_{i=0}^{\left\lfloor \frac{n-s(p-1)-1}{sp+1}\right\rfloor}\binom{n-s(p-1)-1-spi}{i}\ =\ \sum_{i=0}^{\left\lfloor \frac{n-2-s(p-1)}{sp+1}\right\rfloor}\binom{n-s(p-1)-2-spi}{i}+\\
\sum_{i=0}^{\left\lfloor \frac{n-2-sp-s(p-1)}{sp+1}\right\rfloor}\binom{n-2-sp-s(p-1)-spi}{i}.
\end{align*}
We proceed by case analysis.

Case 1: $n = m(sp+1) + r$ with $m\ge 1$ and $1\le r\le s(p-1)$. Then
\begin{align*}
  &\sum_{i=0}^{\left\lfloor \frac{n-s(p-1)-1}{sp+1}\right\rfloor}\binom{n-s(p-1)-1-spi}{i} -   \sum_{i=0}^{\left\lfloor \frac{n-2-s(p-1)}{sp+1}\right\rfloor}\binom{n-s(p-1)-2-spi}{i}\\
  &\ =\ \sum_{i=0}^{m-1} \left(\binom{n-s(p-1)-1-spi}{i}- \binom{n-s(p-1)-2-spi}{i}\right)\\
  &\ =\ \sum_{i=1}^{m-1} \left(\binom{n-s(p-1)-1-spi}{i}- \binom{n-s(p-1)-2-spi}{i}\right)\\
  &\ =\ \sum_{i=1}^{m-1} \binom{n-s(p-1)-2-spi}{i-1}\\
  &\ =\ \sum_{i=0}^{m-2} \binom{n-s(p-1)-2-sp(i+1)}{i}\\
  &\ =\ \sum_{i=0}^{\left\lfloor \frac{n-2-sp-s(p-1)}{sp+1}\right\rfloor}\binom{n-2-sp-s(p-1)-spi}{i},
\end{align*}
because 
$$\left\lfloor \frac{n-2-sp-s(p-1)}{sp+1}\right\rfloor\ =\ \left\lfloor \frac{(m-2)(sp+1)+r+s}{sp+1}\right\rfloor\ =\ m-2.$$

Case 2: $n = m(sp+1) + s(p-1) + 1$ with $m\ge 1$. We have
\begin{align*}
  &\sum_{i=0}^{\left\lfloor \frac{n-s(p-1)-1}{sp+1}\right\rfloor}\binom{n-s(p-1)-1-spi}{i} -   \sum_{i=0}^{\left\lfloor \frac{n-2-s(p-1)}{sp+1}\right\rfloor}\binom{n-s(p-1)-2-spi}{i}\\
  &\ =\ \sum_{i=0}^{m}\binom{n-s(p-1)-1-spi}{i} - \sum_{i=0}^{m-1}\binom{n-s(p-1)-2-spi}{i}\\
  &\ =\ 1 + \sum_{i=1}^{m-1}\left(\binom{n-s(p-1)-1-spi}{i}- \binom{n-s(p-1)-2-spi}{i}\right)\\
  &\ =\ 1 + \sum_{i=1}^{m-1}\binom{n-s(p-1)-2-spi}{i-1}\\
  &\ =\ 1 + \sum_{i=0}^{m-2}\binom{n-s(p-1)-2-sp(i+1)}{i}\\
  &\ =\ \sum_{i=0}^{m-1}\binom{n-s(p-1)-2-sp(i+1)}{i}\\
  &\ =\ \sum_{i=0}^{\left\lfloor \frac{n-2-sp-s(p-1)}{sp+1}\right\rfloor}\binom{n-2-sp-s(p-1)-spi}{i}.
\end{align*}

Case 3: $n = m(sp+1) + r$ with $m\ge 1$ and $s(p-1)+2\le r\le sp$. Then
\begin{align*}
  &\sum_{i=0}^{\left\lfloor\frac{n-s(p-1)-1}{sp+1}\right\rfloor}\binom{n-s(p-1)-1-spi}{i} -   \sum_{i=0}^{\left\lfloor \frac{n-2-s(p-1)}{sp+1}\right\rfloor}\binom{n-s(p-1)-2-spi}{i}\\
  &\ =\ \sum_{i=0}^m \left(\binom{n-s(p-1)-1-spi}{i}  - \binom{n-s(p-1)-2-spi}{i}\right)\\
  &\ =\ \sum_{i=1}^m \left(\binom{n-s(p-1)-1-spi}{i}  - \binom{n-s(p-1)-2-spi}{i}\right)\\
  &\ =\ \sum_{i=1}^m \binom{n-s(p-1)-2-spi}{i-1}\\
  &\ =\ \sum_{i=0}^{m-1} \binom{n-s(p-1)-2-sp(i+1)}{i}\\
  &\ =\ \sum_{i=0}^{\left\lfloor \frac{n-2-sp-s(p-1)}{sp+1}\right\rfloor}\binom{n-2-sp-s(p-1)-spi}{i}.
\end{align*}
\end{proof}

\begin{lem}\label{subseq}
The sequence $(|\mathcal{B}^{(s)}_{p, n}|)_{n=1}^\infty$ is a periodic subsequence of $(b^{(s)}_{p, n})_{n=1}^\infty$.
In particular,
$$|\mathcal{B}^{(s)}_{p, n}| \ =\ b^{(s)}_{p, ns-s+1}.$$ 
\end{lem}
\begin{proof}
    Thanks to \eqref{e20} and Proposition \ref{formulaparentseq}, it suffices to verify that
    $$\sum_{k=1}^{\left\lfloor \frac{ns+1}{ps+1}\right\rfloor}\binom{ns-sp k}{k-1}\ =\ \sum_{i=0}^{\left\lfloor \frac{(n-p)s}{ps+1}\right\rfloor}\binom{ns-sp(i+1)}{i}.$$
    Indeed,
    $$\sum_{k=1}^{\left\lfloor \frac{ns+1}{ps+1}\right\rfloor}\binom{ns-sp k}{k-1}\ =\ \sum_{i=0}^{\left\lfloor \frac{ns+1}{ps+1}\right\rfloor-1}\binom{ns-sp (i+1)}{i}\ =\ \sum_{i=0}^{\left\lfloor \frac{(n-p)s}{ps+1}\right\rfloor}\binom{ns-sp (i+1)}{i}.$$
\end{proof}

\begin{defi}\normalfont
Let $p(x) = c_0 + c_1x + \cdots + c_rx^r$ be a polynomial with real coefficients $(c_i)_{i=0}^r$. A sequence $(a_n)_{n=1}^\infty$ is said to satisfy $p(x)$ if 
$$c_0a_n + c_1a_{n-1} + \cdots + c_ra_{n-r}\ =\ 0,\mbox{ for all }n\ge r+1.$$
\end{defi}

By definition, the sequence $(b^{(s)}_{p, n})_{n=1}^\infty$ satisfies the polynomial
$$u_{s, p}(x)\ :=\ 1 - x - x^{ps + 1}.$$

\begin{lem}\cite[Lemma 2.3]{CV}\label{relatepcseq}
Given two polynomials $p(x)$ and $q(x)$ such that $p(x)$ divides $q(x)$, if a sequence $(a_n)_{n=1}^\infty$ satisfies $p(x)$, then $(a_n)_{n=1}^\infty$ satisfies $q(x)$.
\end{lem}

\begin{proof}[Proof of Theorem \ref{m2}]
Theorem \ref{m2} states that $(|\mathcal{B}^{(s)}_{p, n}|)_{n=1}^\infty$ satisfies
$$v_{s, p}(x)\ :=\ 1 - \binom{s}{0}x - \binom{s}{1}x^{1+p} - \cdots - \binom{s}{s}x^{1+sp}\ =\ 1 - \sum_{i=0}^s\binom{s}{i}x^{1+ip}.$$
We have 
\begin{align*}
    u_{s,p}(x)\cdot \sum_{i=0}^{s-1}x^i(x^{sp}+1)^i&\ =\ (1 - x - x^{ps + 1})\frac{1-x^s(x^{sp}+1)^s}{1- x(x^{sp}+1)}\\
    & \ =\ 1 - x^s(x^{sp}+1)^s\\
    &\ =\ 1 - x^s\sum_{i=0}^s \binom{s}{i} x^{spi}\\
    &\ =\ 1 - \sum_{i=0}^s \binom{s}{i}(x^s)^{1+ ip}\ =\ v_{s, p}(x^s).
\end{align*}
Hence,  $u_{s,p}(x)$ divides $v_{s,p}(x^s)$. By Lemma \ref{relatepcseq}, $(b^{(s)}_{p, n})_{n=1}^\infty$ satisfies 
$v_{s,p}(x^s)$. 

Due to the power $s$ in $v_{s,p}(x^s)$, $v_{s,p}(x^s)$ gives a linear recurrence for terms in $(b^{(s)}_{p, n})_{n=1}^\infty$ whose indices are $s$ apart. It follows from Lemma \ref{subseq} that the sequence $(|\mathcal{B}_{p, n}|)_{n=1}^\infty$ satisfies $v_{s, p}(x)$. 
\end{proof}

\section{Schreier sets that do not admit multiples of a given number}\label{notmul}
In this section, given $u\ge 2$, we count Schreier sets that do not contain any multiple of $u$. This is the opposite of what was studied in \cite{CV}. First, we need a formula that gives exactly numbers that are not divisible by $u$.

\begin{lem}
    For $u\ge 2$, we have
    $$\left\{\left\lfloor\frac{un-1}{u-1}\right\rfloor\,:\, n\in \mathbb{N}\right\}\ =\ \{n\in\mathbb{N}\,:\, u\nmid n\}.$$
\end{lem}

\begin{proof}
    Write $n = (u-1)j + m$ with $j\ge 0$ and $0\le m\le u-2$. Then
    \begin{align*}
        \left\lfloor\frac{un-1}{u-1}\right\rfloor\ =\ \left\lfloor \frac{u((u-1)j + m)-1}{u-1}\right\rfloor&\ =\ uj+m+\left\lfloor \frac{m-1}{u-1}\right\rfloor\\
        &\ =\ \begin{cases}uj  - 1, &\mbox{ if }m = 0;\\
        uj + m, &\mbox{ if } 1\le m \le u-2.\end{cases}
    \end{align*}
    This shows that the formula $\lfloor (un-1)/(u-1)\rfloor$ gives exactly positive integers that are not divisible by $u$.
\end{proof}

We prove Theorem \ref{m3} using the same method of characteristic polynomials as in Section \ref{colored}. First, 
we find a formula for $|\mathcal{D}_{u, n}|$. A $k$-element in $\mathcal{D}_{u, n}$ is uniquely determined by $(k-1)$ elements in
$$\left\{\left\lfloor \frac{u(j_k-1)-1}{u-1}\right\rfloor, \ldots, \left\lfloor\frac{u(n-1)-1}{u-1}\right\rfloor\right\},$$
where $j_k$ is the smallest positive integer with
$$\frac{u(j_k-1)-1}{u-1}\ \ge\ k\ \Longleftrightarrow\ j_k\ = \ k+1 - \left\lfloor \frac{k-1}{u}\right\rfloor.$$
Hence, the number of $k$-element sets in $\mathcal{D}_{u,n}$ is
$$\binom{n-j_k+1}{k-1}\ =\ \binom{n-k+\left\lfloor \frac{k-1}{u}\right\rfloor}{k-1}$$
under the condition 
$$n-k+\left\lfloor \frac{k-1}{u}\right\rfloor \ \ge\ k-1\ \Longleftrightarrow\ k\ \le\ \frac{(n+1)u-1}{2u-1}.$$
Therefore,
\begin{equation}\label{ee1}|\mathcal{D}_{u,n}|\ =\ \sum_{k=1}^{\left\lfloor\frac{(n+1)u-1}{2u-1}\right\rfloor} \binom{n-k+\left\lfloor \frac{k-1}{u}\right\rfloor}{k-1}.\end{equation}

We now define parent sequences $(d_{u, n})_{n=1}^\infty$ so that for each fixed $u$, $(|\mathcal{D}_{u, n}|)_{n=1}^\infty$ is a periodic subsequence of $(d_{u,n})_{n=1}^\infty$. Let 
$$d_{u, 1}\ =\ d_{u, 2}\ =\ \cdots \ \ =\ d_{u,2u-1}\ =\ 1$$
and 
$$d_{u, n}\ =\ d_{u, n-u} + d_{u, n-2u+1}\mbox{ for }n\ge 2u.$$

For example, 
\begin{align*}
 (d_{2,n})_{n=1}^\infty: \quad &1, 1, 1, 2, 2, 3, 4, 5, 7, 9, 12, 16, 21, 28, 37, 49, 65, 86, 114,
151, 200, \\
&265, 351, \ldots;\\
 (d_{3,n})_{n=1}^\infty: \quad &1, 1, 1, 1, 1, 2, 2, 2, 3, 3, 4, 5, 5, 7, 8, 9, 12, 13, 16, 20, 22, 
28, 33, 38, 48, \\
&55, 66, 81, 93, 114, 136, 159, 195, 229,\ldots;\\
(d_{4,n})_{n=1}^\infty: \quad &1, 1, 1, 1, 1, 1, 1, 2, 2, 2, 2, 3, 3, 3, 4, 5, 5, 5, 7, 8, 8, 9, 12,
13, 13, 16, 20, 21,\\
&22, 28, 33, 34, 38, 48, 54, 56, 66, 81, 88, 94, 114, 135, 144, 160, 195, \ldots.\\
\end{align*}

\begin{prop}\label{l1}    For $u\ge 2$ and $n\in\mathbb{N}$, we have
    $$d_{u,n}\ =\ \sum_{i=0}^{\left\lfloor \frac{n-1}{u}\right\rfloor}\binom{\left\lfloor\frac{n+(u-1)i-u}{2u-1}\right\rfloor}{i}.$$
\end{prop}

\begin{proof}
    Let $u\ge 2$. 
    If $1\le n\le u$, then
    $\left\lfloor \frac{n-1}{u}\right\rfloor = 0$, and
    $$\sum_{i=0}^{\left\lfloor \frac{n-1}{u}\right\rfloor}\binom{\left\lfloor\frac{n+(u-1)i-u}{2u-1}\right\rfloor}{i}\ =\  \binom{\left\lfloor\frac{n-u}{2u-1}\right\rfloor}{0}\ =\ 1.
    $$
    
    If $u+1\le n\le 2u-1$, then
    $$1\ \le\ \frac{n-1}{u}\ \le\ \frac{2u-2}{u}\ =\ 2 - \frac{2}{u}\ <\ 2\ \Longrightarrow\ \left\lfloor\frac{n-1}{u}\right\rfloor \ =\ 1,$$
    and
      \begin{align*}\sum_{i=0}^{\left\lfloor \frac{n-1}{u}\right\rfloor}\binom{\left\lfloor\frac{n+(u-1)i-u}{2u-1}\right\rfloor}{i}&\ =\ \sum_{i=0}^{1}\binom{\left\lfloor\frac{n+(u-1)i-u}{2u-1}\right\rfloor}{i}\\
    &\ =\ \binom{\left\lfloor\frac{n-u}{2u-1}\right\rfloor}{0} + \binom{\left\lfloor\frac{n-1}{2u-1}\right\rfloor}{1}\\
    &\ =\ \binom{0}{0} + \binom{0}{1}\ =\ 1.
    \end{align*}

    It remains to verify that $d_{u, n} = d_{u, n-u} + d_{u, n-2u+1}$ for all $n\ge 2u$, i.e., 
    $$\sum_{i=0}^{\left\lfloor \frac{n-1}{u}\right\rfloor}\binom{\left\lfloor\frac{n+(u-1)i-u}{2u-1}\right\rfloor}{i}\\
    \ =\ \sum_{i=0}^{\left\lfloor \frac{n-1}{u}\right\rfloor-1}\binom{\left\lfloor\frac{n+(u-1)i-2u}{2u-1}\right\rfloor}{i} + \sum_{i=0}^{\left\lfloor \frac{n}{u}\right\rfloor-2}\binom{\left\lfloor\frac{n+(u-1)i-u}{2u-1}\right\rfloor-1}{i}.$$

By Pascal's rule, we have
\begin{align*}
&\sum_{i=0}^{\left\lfloor \frac{n-1}{u}\right\rfloor}\binom{\left\lfloor\frac{n+(u-1)i-u}{2u-1}\right\rfloor}{i}\\
\ =\ &1  + \sum_{i = 1}^{\left\lfloor \frac{n-1}{u}\right\rfloor}\binom{\left\lfloor\frac{n+(u-1)i-u}{2u-1}\right\rfloor}{i}\\
\ =\  &1 + \underbrace{\sum_{i=1}^{\left\lfloor \frac{n-1}{u}\right\rfloor} \binom{\left\lfloor\frac{n+(u-1)i-u}{2u-1}\right\rfloor-1}{i}}_{=: I(u, n)} + \underbrace{\sum_{i=1}^{\left\lfloor \frac{n-1}{u}\right\rfloor} \binom{\left\lfloor\frac{n+(u-1)i-u}{2u-1}\right\rfloor-1}{i-1}}_{=: II(u,n)}.
\end{align*}
Reindexing $II(u, n)$ gives 
$$
II(u, n) \ =\ \sum_{i=0}^{\left\lfloor \frac{n-1}{u}\right\rfloor - 1} \binom{\left\lfloor\frac{n+(u-1)(i+1)-u-2u+1}{2u-1}\right\rfloor}{i} = \sum_{i=0}^{\left\lfloor \frac{n-1}{u}\right\rfloor -1} \binom{\left\lfloor\frac{n+(u-1)i-2u}{2u-1}\right\rfloor}{i}.
$$
Hence, it suffices to show that
$$\underbrace{\sum_{i=1}^{\left\lfloor \frac{n-1}{u}\right\rfloor} \binom{\left\lfloor\frac{n+(u-1)i-u}{2u-1}\right\rfloor-1}{i}}_{=: I(u, n)}\ =\ \sum_{i=1}^{\left\lfloor \frac{n}{u}\right\rfloor-2}\binom{\left\lfloor\frac{n+(u-1)i-u}{2u-1}\right\rfloor-1}{i},$$
which means that
$$\binom{\left\lfloor \frac{n+(u-1)i - u}{2u-1}\right\rfloor - 1}{i}\ =\ 0, \mbox{ for }\left\lfloor \frac{n}{u}\right\rfloor - 1 \le i\le \left\lfloor \frac{n-1}{u}\right\rfloor.$$
Write $n=qu+r$, where $q\ge 2$ and $0\le r \le u-1$. Then 
$$
\left\lfloor \frac{n}{u}\right\rfloor - 1 \ =\ q - 1\quad \mbox{ and }\quad \left\lfloor \frac{n-1}{u} \right\rfloor\ =\
\begin{cases}
    q, &\mbox{if } r \ge 1;\\
    q-1, & \mbox{if } r = 0.
\end{cases}
$$

Case 1: if $i=q-1$, then 
$$
\left\lfloor \frac{n + (u-1)(q-1) - u}{2u-1} \right\rfloor
\ =\ \left\lfloor \frac{(2u-1)q + r - (2u-1)}{2u-1} \right\rfloor
\ =\ q-1.
$$
Hence $$\binom{\left\lfloor \frac{n+(u-1)i-u}{2u-1} \right\rfloor-1}{q-1}\ =\ 0.$$

Case 2: if $i = q$, then
$$
\left\lfloor \frac{n + (u-1)q - u}{2u-1} \right\rfloor
\ =\ \left\lfloor \frac{(2u-1)q+r - u}{2u-1} \right\rfloor
\ =\  q-1.
$$
Thus
$$
\binom{\left\lfloor \frac{n+(u-1)i-u}{2u-1} \right\rfloor-1}{q}\ =\ 0.
$$
This completes our proof. 
\end{proof}

\begin{prop}\label{ep1}
For each $u\ge 2$, the sequence $(|\mathcal{D}_{u,n}|)_{n=1}^\infty$ is a periodic subsequence of $(d_{u,n})_{n=1}^\infty$. In particular,
$$|\mathcal{D}_{u,n}| \ =\ d_{u, un-(u-1)}, \mbox{ for all }n\in \mathbb{N}.$$
\end{prop}

\begin{proof}Fix $u\ge 2$ and $n\in \mathbb{N}$.
    Thanks to \eqref{ee1} and Proposition \ref{l1}, it suffices to prove that
    $$\sum_{k=0}^{\left\lfloor\frac{nu-u}{2u-1}\right\rfloor}\underbrace{\binom{\left\lfloor \frac{un - (u-1)k}{u}\right\rfloor-1}{k}}_{=: f(k)}\ =\ \sum_{j=0}^{n-1}\underbrace{\binom{\left\lfloor\frac{un+(u-1)j}{2u-1}\right\rfloor-1}{j}}_{=:g(j)}.$$

    If $n = 1$, both sides equal $1$. Suppose that $n\ge 2$. For $0\le j\le n-1$, we consider $j$ of the form
    $$j \ =\ n - (2u-1)(\ell+1)\ =:\ \xi(\ell).$$
    Then 
    $$0\ \le\ \ell\ \le\ \underbrace{\left\lfloor\frac{n-2u}{2u-1} + \frac{1}{2u-1}\right\rfloor}_{=:\ell_0}\ <\ \frac{n-2u}{2u-1} + \frac{1}{u}.$$
    For $\ell$ in the above range, if $k = \chi(\ell) := u(\ell+1)-1$, then $0\le k\le \left\lfloor\frac{nu-u}{2u-1}\right\rfloor$.
    We have 
    \begin{align}\label{ee2}g(\xi(\ell))\ =\ g(n-(2u-1)(\ell+1))&\ =\ \binom{n + \ell - (\ell+1)u}{n-(2u-1)(\ell+1)}\nonumber\\
    &\ =\ \binom{n + \ell - (\ell+1)u}{(\ell+1)u-1}\nonumber\\
    &\ =\ f(u\ell + (u-1))\ =\ f(\chi(\ell)).
    \end{align}
    Write
    \begin{align*}\sum_{k=0}^{\left\lfloor\frac{nu-u}{2u-1}\right\rfloor} f(k)\ =\ &\sum_{k=0}^{\chi(0)-1}f(k) + f(\chi(0)) + \sum_{k = \chi(0)+1}^{\chi(1) - 1}f(k) + f(\chi(1)) + \cdots + \\ 
    &\sum_{k = \chi(\ell_0-1)+1}^{\chi(\ell_0) - 1}f(k) + f(\chi(\ell_0)) + \sum_{k=\chi(\ell_0)+1}^{\left\lfloor\frac{nu-u}{2u-1}\right\rfloor}f(k).
    \end{align*}
    and
    \begin{align*}
    \sum_{j=0}^{n-1} g(j)\ =\ &\sum_{j=0}^{\xi(\ell_0)-1}g(j) + g(\xi(\ell_0)) + \sum_{j = \xi(\ell_0)+1}^{\xi(\ell_0-1) - 1}g(j) + g(\xi(\ell_0-1)) + \cdots + \\ 
    &\sum_{j = \xi(1)+1}^{\xi(0) - 1}g(j) + g(\xi(0)) + \sum_{j=\xi(0)+1}^{n-1}g(j).\end{align*}
    Note that \eqref{ee2} gives $f(\chi(\ell)) = g(\xi(\ell))$ for each $\ell\in [0, \ell_0]$.

    Next, we prove that for $\ell\in [0, \ell_0-1]$,
    \begin{equation}\label{ee3}\sum_{k = \chi(\ell)+1}^{\chi(\ell+1) - 1}f(k)\ =\ \sum_{j = \xi(\ell+1)+1}^{\xi(\ell)-1}g(j).\end{equation}
    To do so, note that the right sum of \eqref{ee3} has 
    $2u-2$ terms, while the left sum has $u-1$ terms. To prove \eqref{ee3}, we use the Pascal's rule to combine every two consecutive terms in the right sum. On the one hand, we have
    \begin{align*}
        &\sum_{j = \xi(\ell+1)+1}^{\xi(\ell)-1}g(j)\\
        \ =\ &\sum_{i = 1}^{u-1}(g(\xi(\ell+1)+2i-1) + g(\xi(\ell+1)+2i))\\
        \ =\ &\sum_{i=1}^{u-1}\left(\binom{n-(u-1)(\ell+2)+i-2}{\xi(\ell+1)+2i-1} + \binom{n-(u-1)(\ell+2)+i-2}{\xi(\ell+1)+2i}\right)\\
        \ =\ &\sum_{i=1}^{u-1}\binom{n-(u-1)(\ell+2)+i-1}{\xi(\ell+1)+2i}\ =\ \sum_{i=1}^{u-1}\binom{n-(u-1)(\ell+2)+i-1}{u(\ell+2)-i-1}.
    \end{align*}
    On the other hand, 
    \begin{align*}
        \sum_{k=\chi(\ell) + 1}^{\chi(\ell+1)-1} f(k)&\ =\ \sum_{i=1}^{u-1} f(\chi(\ell+1)-i)\\
        &\ =\ \sum_{i=1}^{u-1}\binom{\left\lfloor \frac{un-(u-1)(\chi(\ell+1)-i)}{u}\right\rfloor-1}{\chi(\ell+1)-i}\\
        &\ =\ \sum_{i=1}^{u-1}\binom{n-(u-1)(\ell+2)+i-1}{u(\ell+2)-i-1}.
    \end{align*}
Therefore, \eqref{ee3} holds. 

    From \eqref{ee2} and \eqref{ee3}, we have shown that 
    $$\sum_{k = \chi(0)}^{\chi(\ell_0)}f(k)\ =\ \sum_{j = \xi(\ell_0)}^{\xi(0)} g(j).$$
    It remains to verify that
    \begin{equation}\label{ee4}
        \sum_{k=0}^{\chi(0)-1} f(k) \ =\ \sum_{j= \xi(0) + 1}^{n-1} g(j)
    \end{equation}
    and
    \begin{equation}\label{ee5}
       \sum_{k = \chi(\ell_0) + 1}^{\left\lfloor\frac{nu-u}{2u-1}\right\rfloor}f(k)\ =\  \sum_{j = 0}^{\xi(\ell_0) - 1}g(j).
    \end{equation}
    The proofs of \eqref{ee4} and \eqref{ee5} use a similar idea to the proof of \eqref{ee3}, so we move it to the appendix. 
\end{proof}

\begin{proof}[Proof of Theorem \ref{m3}]
The first statement follows from the well-known formula for $F_n$:
\begin{equation}\label{oot}F_n\ =\ \sum_{i=0}^{\left\lfloor \frac{n-1}{2}\right\rfloor}\binom{n-i-1}{i}.\end{equation}
For a proof of \eqref{oot} using tilings, see \cite[Identity 4]{BQ}, or if we assume Zeckendorf's Theorem\footnote{Zeckendorf's Theorem states that every positive integer can be written uniquely as a sum of nonadjacent Fibonacci numbers in $(F_n)_{n\ge 2}$.}, \cite[(2.5)]{KKMW} gives another proof. 
Combining \eqref{ee1} with \eqref{oot}, we need to show the equality
\begin{equation}\label{ee10}\sum_{k=1}^{\left\lfloor\frac{(n+1)u-1}{2u-1}\right\rfloor} \binom{n-k+\left\lfloor \frac{k-1}{u}\right\rfloor}{k-1}\ =\ \sum_{i=0}^{\left\lfloor \frac{n-1}{2}\right\rfloor}\binom{n-i-1}{i}\mbox{ for }n\le 2u-1,\end{equation}
whose proof is in Section \ref{apen}.

By definition, the sequence $(d_{u,n})_{n=1}^\infty$ satisfies the polynomial
$$p_u(x)\ :=\ 1-x^{u}-x^{2u-1}.$$
Theorem \ref{m3} states that $(|\mathcal{D}_{k, n}|)_{n=1}^\infty$ satisfies 
$$
        q_u(x) \ :=\ 1+\sum_{i=1}^{u} \binom{u}{i} (-x)^{i} - x^{2u-1}\ =\ (1-x)^u - x^{2u-1}.
$$
By Lemma 3.4 and Proposition \ref{ep1}, it suffices to verify that 
$p_u(x)$ divides $q_u(x^u)$. We have
\begin{align*}
&(1-x^u-x^{2u-1})\left(\sum_{j=0}^{u-1}(1-x^u)^{u-1-j}x^{j(2u-1)}\right)\\
\ =\ &(1-x^u-x^{2u-1}) \frac{(1-x^u)^{u}-x^{u(2u-1)}}{1-x^u-x^{2u-1}}\\
\ =\ &(1-x^u)^{u} - x^{u(2u-1)} \ =\ q_u(x^u). 
\end{align*}
\end{proof}

\section{Further investigations}\label{future}
One may consider other setups for $p, q$ and $(s_{n,i})_{i=1}^n$ in Problem \ref{mainprob} and determine what recurrences they give. Interested readers may also study the more general condition $q\min F\ge p|F|$ as in \cite{BCF}. Define
$$\mathcal{A}^{(s)}_{p, q, n}\ :=\ \{F\subset \{\underbrace{1, \ldots, 1}_{s}, \ldots, \underbrace{n-1, \ldots, n-1}_{s}, n\}\,:\,n\in F\mbox{ and }q\min F\ge p|F|\}.$$
Similar to \eqref{e2}, we have a formula for $|\mathcal{A}^{(s)}_{n, p, q}|$:
$$|\mathcal{A}^{(s)}_{p, q, n}|\ =\ \sum_{k=1}^{\left\lfloor \frac{nsq + q}{sp + q}\right\rfloor}\sum_{j=0}^{\left\lfloor \frac{k-1}{s+1}\right\rfloor}(-1)^j\binom{n-\left\lceil \frac{pk}{q}\right\rceil}{j}\binom{n - \left\lceil \frac{pk}{q}\right\rceil + k - 1- (s+1)j-1}{n-\left\lceil \frac{pk}{q}\right\rceil -1}.$$

Below are the data and conjectured recurrences we gather
\begin{enumerate}
    \item $(|\mathcal{A}^{(2)}_{1, 2, n}|)_{n=1}^\infty$: $1, 2, 4, 9, 19, 41, 88, 189, 406, 872, 1873, 4023, 8641, \ldots$ with
    $$a_n \ =\ a_{n-1} + 2a_{n-2} + a_{n-3};$$
    \item $(|\mathcal{A}^{(2)}_{1, 3, n}|)_{n=1}^\infty$: $1, 3, 6, 13, 31, 73, 169, 392, 912, 2121, 4930, 11460, 26642,\ldots$ with
    $$a_n \ =\ 3a_{n-1} - 3a_{n-2} + 4a_{n-3}-2a_{n-4}+a_{n-5};$$
    \item $(|\mathcal{A}^{(2)}_{1,4, n}|)_{n=1}^\infty$: 
    $$1, 3, 8, 18, 41, 100, 250, 617, 1501, 3643, 8877, 21689, 52984, 129303,\ldots\mbox{ with}$$ 
    $$a_n = 3a_{n-1} - 3a_{n-2} + 3a_{n-3} + 2a_{n-4} + a_{n-5};$$
    \item   $(|\mathcal{A}^{(2)}_{1,5, n}|)_{n=1}^\infty$: 
    $$1, 3, 9, 23, 54, 127, 314, 808, 2090, 5326, 13379, 33460, 83979, 
211847,\ldots\mbox{ with}$$ 
    $$a_n = 5a_{n-1} - 10a_{n-2} + 10a_{n-3} - 4a_{n-5} + a_{n-6} + a_{n-7};$$
    \item $(|\mathcal{A}^{(2)}_{3, 2, n}|)_{n=1}^\infty$: $0,1,1,2, 3, 5, 8, 14, 24, 40, 66,110,185, 311, 521, 871,\ldots$ with
    $$a_n = a_{n-1}+2a_{n-4} + a_{n-5} + a_{n-6} + a_{n-7}.$$
\end{enumerate}

\begin{prob}\label{fi}\normalfont
For $(s, p, q)\in \mathbb{N}^3$, find the recurrence for $(|\mathcal{A}^{(s)}_{p, q, n}|)_{n=1}^\infty$.    
\end{prob}

Observe that \cite[Theorem 1.1]{BCF} and Theorem \ref{m1} solve Problem \ref{fi} when $s=1$ and $q = 1$, respectively.

Last but not least, we are interested in nonlinear Schreier-type conditions $(\min F)^q \ge |F|^p$ as started in \cite{CIMSZ}.

\appendix
\section{Technical proofs}\label{apen}

\begin{proof}[Proof of \eqref{ee4}] We have
    \begin{align*}
        \sum_{k=0}^{\chi(0)-1}f(k)&\ =\ \sum_{k=0}^{u-2}\binom{\left\lfloor \frac{un - (u-1)k}{u}\right\rfloor-1}{k}\\
        &\ =\ \sum_{k=0}^{u-2}\binom{n-k-1+\left\lfloor \frac{k}{u}\right\rfloor}{k}\ =\ \sum_{k=0}^{u-2}\binom{n-k-1}{k},
    \end{align*}
    while
    \begin{align*}
        &\sum_{j = \xi(0)+1}^{n-1}g(j)\\
        \ =\ &\sum_{j=n-2u+2}^{n-1}\binom{\left\lfloor\frac{un+(u-1)j}{2u-1}\right\rfloor - 1}{j}\\
        \ =\ &\sum_{j = n-2u+2}^{n-3}\binom{\left\lfloor\frac{un+(u-1)j}{2u-1}\right\rfloor - 1}{j} + 1\\
        \ =\ &\sum_{i=1}^{u-2}\left(\binom{\left\lfloor\frac{un+(u-1)(n-2u+2i)}{2u-1}\right\rfloor - 1}{n-2u+2i} +
        \binom{\left\lfloor\frac{un+(u-1)(n-2u+2i+1)}{2u-1}\right\rfloor - 1}{n-2u+2i+1}\right) + 1\\
        \ =\ &\sum_{i=1}^{u-2}\left(\binom{n-u+i-1 + \left\lfloor \frac{u-i}{2u-1}\right\rfloor}{n-2u+2i} + \binom{n-u+i+\left\lfloor\frac{-i}{2u-1} \right\rfloor}{n-2u+2i+1}\right) + 1\\
        \ =\ &\sum_{i=1}^{u-2}\left(\binom{n-u+i-1}{n-2u+2i} + \binom{n-u+i-1}{n-2u+2i+1}\right) + 1\\
        \ =\ &\sum_{i=1}^{u-2}\binom{n-u+i}{n-2u+2i+1} + 1\\
        \ =\ &\sum_{i=1}^{u-2}\binom{n-u+i}{u-i-1}+1\\
        \ =\ &\sum_{k=0}^{u-2}\binom{n-k-1}{k}\mbox{ by setting } u = k+i+1.
    \end{align*}
    Hence, \eqref{ee4} holds. 
\end{proof}

\begin{proof}[Proof of \eqref{ee5}]
    We have
    $$\xi(\ell_0)-1\ =\ n-1 - (2u-1)\left\lfloor\frac{n}{2u-1}\right\rfloor\mbox{ and }\chi(\ell_0) + 1 \ =\ u\left\lfloor\frac{n}{2u-1}\right\rfloor.$$
    Hence, \eqref{ee5} is the same as
    \begin{equation}\label{eee6}\sum_{k = u\left\lfloor\frac{n}{2u-1}\right\rfloor}^{\left\lfloor\frac{nu-u}{2u-1}\right\rfloor}\binom{\left\lfloor\frac{un-(u-1)k}{u}\right\rfloor-1}{k}\ =\ \sum_{j=0}^{n-1-(2u-1)\left\lfloor\frac{n}{2u-1}\right\rfloor}\binom{\left\lfloor\frac{un+(u-1)j}{2u-1}\right\rfloor - 1}{j}.\end{equation}
    Write $n = (2u-1)s + t$ with $s\ge 0$ and $0\le t\le 2u-2$. If $t = 0$, both sides of \eqref{eee6} equals $0$. Assume that $1\le t\le 2u-2$ and  rewrite \eqref{eee6} as
    \begin{equation}\label{eee7}\sum_{k = us}^{us + \left\lfloor\frac{(t-1)u}{2u-1}\right\rfloor}\binom{2us + t - k - 1}{k}\ =\ \sum_{j=0}^{t-1}\binom{su + \left\lfloor \frac{tu + (u-1)j}{2u-1}\right\rfloor - 1}{j}.\end{equation}

    Case 1: $t = 2r + 1$ for some $r\ge 0$. Then $r\le u-2$. We have
    \begin{align*}
        &\sum_{j=0}^{t-1}\binom{su + \left\lfloor \frac{tu + (u-1)j}{2u-1}\right\rfloor - 1}{j}\\
        \ =\ &\sum_{j=0}^{2r}\binom{su + \left\lfloor \frac{(2r+1)u + (u-1)j}{2u-1} \right\rfloor-1}{j}\\
        \ =\ &1 + \sum_{i=1}^r\left(\binom{su + \left\lfloor \frac{(2r+1)u+(u-1)(2i-1)}{2u-1}\right\rfloor-1}{2i-1} + \binom{su + \left\lfloor \frac{(2r+1)u+(u-1)2i}{2u-1}\right\rfloor - 1}{2i}\right)\\
        \ =\ &1 + \sum_{i=1}^r\left(\binom{su + r + i + \left\lfloor \frac{r-i+1}{2u-1}\right\rfloor-1}{2i-1} + \binom{su + r + i + \left\lfloor \frac{r-i+u}{2u-1}\right\rfloor - 1}{2i}\right)\\
        \ =\ &1 + \sum_{i=1}^r\left(\binom{su + r + i-1}{2i-1} + \binom{su + r + i  - 1}{2i}\right)\\
        \ =\ &1 + \sum_{i=1}^r\binom{su + r+ i}{2i}
    \end{align*}
    and
    \begin{align*}
    \sum_{k = us}^{us + \left\lfloor\frac{(t-1)u}{2u-1}\right\rfloor}\binom{2us + t - k - 1}{k}&\ =\ \sum_{k = us}^{us + r}\binom{2us + 2r - k}{k}\\
    &\ =\ \sum_{k=us}^{us + r}\binom{2us + 2r - k}{2us + 2r - 2k}\\
    &\ =\ \sum_{j = 0}^r\binom{us + r + j}{2j}\mbox{ by setting }j = r - k + us.
    \end{align*}
    Hence, \eqref{eee6} holds in this case. 

    Case 2: $t = 2r$ for $r\ge 1$. Then $r\le u-1$. We have
    \begin{align*}
        &\sum_{j=0}^{t-1}\binom{su + \left\lfloor \frac{2ru + (u-1)j}{2u-1}\right\rfloor - 1}{j}\\
        \ =\ &\sum_{j=0}^{2r-1}\binom{su + \left\lfloor \frac{2ru + (u-1)j}{2u-1}\right\rfloor-1}{j}\\
        \ =\ &\sum_{i = 0}^{r-1}\left(\binom{su + \left\lfloor \frac{2ru + (u-1)2i}{2u-1}\right\rfloor - 1}{2i} + \binom{su + \left\lfloor \frac{2ru + (u-1)(2i+1)}{2u-1}\right\rfloor-1}{2i+1}\right)\\
        \ =\ &\sum_{i = 0}^{r-1}\left(\binom{su + r + i + \left\lfloor \frac{r-i}{2u-1}\right\rfloor - 1}{2i} + \binom{su + r + i + \left\lfloor \frac{r + u - i - 1}{2u-1}\right\rfloor-1}{2i+1}\right)\\
        \ =\ &\sum_{i = 0}^{r-1}\left(\binom{su + r + i- 1}{2i} + \binom{su + r + i -1}{2i+1}\right)\\
        \ =\ &\sum_{i= 0}^{r-1}\binom{su + r+ i}{2i+1}
    \end{align*}
    and 
    \begin{align*}
        \sum_{k = us}^{us + \left\lfloor\frac{(t-1)u}{2u-1}\right\rfloor}\binom{2us + t - k - 1}{k}&\ =\ \sum_{k = us}^{us + r-1} \binom{2us + 2r - k - 1}{k}\\
        &\ =\ \sum_{k=us}^{us + r - 1}\binom{2us + 2r - k - 1}{2us + 2r - 2k - 1}\\
        &\ =\ \sum_{j=0}^{r-1}\binom{us + r + j}{2j+1}\mbox{ by setting }j = us + r -k - 1.
    \end{align*}
    Thus, \eqref{eee6} holds in this case as well. 
\end{proof}

\begin{proof}[Proof of \eqref{ee10}]
    Since $n\le 2u-1$, 
    $$\frac{k-1}{u}\ \le\ \frac{\frac{(n+1)u-1}{2u-1}-1}{u}\ =\ \frac{(n-1)u}{u(2u-1)}\ =\ \frac{n-1}{2u-1}\ \le\ \frac{2u-2}{2u-1},$$
    so $\left\lfloor\frac{k-1}{u}\right\rfloor = 0$. Hence
    $$\sum_{k=1}^{\left\lfloor\frac{(n+1)u-1}{2u-1}\right\rfloor} \binom{n-k+\left\lfloor \frac{k-1}{u}\right\rfloor}{k-1}\ =\ \sum_{k=1}^{\left\lfloor\frac{(n+1)u-1}{2u-1}\right\rfloor} \binom{n-k}{k-1}\ =\ \sum_{j=0}^{\left\lfloor\frac{(n-1)u}{2u-1}\right\rfloor} \binom{n-j-1}{j}.$$
    It remains to verify that
    \begin{equation}\label{e11}\left\lfloor\frac{(n-1)u}{2u-1}\right\rfloor\ =\ \left\lfloor\frac{n-1}{2}\right\rfloor, \mbox{ for }1\le n\le 2u-1.\end{equation}
    
    If $n = 2r-1$ with $1\le r\le u$, then 
    $$ \frac{(n-1)u}{2u-1}\ =\ \frac{2u}{2u-1}(r-1)\ \in\ [r-1, r).$$

    If $n = 2r$ with $1\le r\le u-1$, then
    $$\frac{(n-1)u}{2u-1}\ =\ \frac{(2r-1)u}{2u-1}\ \in\ [r-1, r).$$

    In both cases, we have \eqref{e11}.
\end{proof}

%%%%%%%%%%%%%%%%%%%%%%%%%%%%%%%%%%%%%%%%%%%%%%%%%%%%%%%%%%%%%%%%%%%%%%%%%%%%%%%%%%%%%%%%%%%%%%%%%%%%%%%%%%%%%%%%%%%%%%%%%%%%%%%%%%%%%%%%%%%%%%%%%%%%%%%%%%%%%%%%%%%%%%%%%%%%%%%%%%%%%%%%%%%%%%%%%%%%%%%%%%%%%%%%%%%%%%%%%%%%%%%%%%%%%%%%%%%%%%%%%%%%%%%%%%%%%%%%%%%%%%%%%%%%%%%%%%%%%%%%%%%%%%%%%%%%%%%%%%%%%%%%%%%%%%%%%%%%%%%%%%%%%%%%%%%%%%%%%%%%%%%%%%%%%%%%%%%%%%%%%%%

%%%%%%%%%%%%%%%%%%%%%%%%%%%%%%%%%%%%%%%%%%%%%%%%%%%%%%%%%%%%%%%%%%%%%%%%%%%%%%%%%%%%%%%%%%%%%%%%%%%%%%%%%%%%%%%%%%%%%%%%%%%%%%%%%%%%%%%%%%%%%%%%%%%%%%%%%%%%%%%%%%%%%%%%%%%%%%%%%%%%%%%%%%%%%%%%%%%%%%%%%%%%%%%%%%%%%%%%%%%%%%%%%%%%%%%%%%%%%%%%%%%%%%%%%%%%%%%%%%%%%%%%%%%%%%%%%%%%%%%%%%%%%%%%%%%%%%%%%%%%%%%%%%%%%%%%%%%%%%%%%%%%%%%%%%%%%%%%%%%%%%%%%%%%%%%%%%%%%%%%%%%%%%%%%%%%

\ \\
\end{document}